\documentclass[11pt]{amsart}
\usepackage[utf8]{inputenc}
\usepackage[pdftex]{graphicx}
\usepackage{amssymb}
\usepackage{amsmath}
\usepackage{amsfonts}
\usepackage{amsthm}
\usepackage{bm}
\usepackage{mathrsfs}
\usepackage{appendix}
\usepackage{enumerate}
\usepackage[left=3cm, right=3cm, bottom=3.4cm]{geometry}
\usepackage[stretch=30,shrink=30]{microtype}
\usepackage{stmaryrd}
\usepackage[normalem]{ulem} 
\usepackage{mathtools}
\usepackage{todonotes}

\usepackage{xcolor}
\definecolor{antiquefuchsia}{rgb}{0.57, 0.36, 0.51}
\definecolor{azure}{rgb}{0.0, 0.5, 1.0}

\usepackage[colorlinks = true, linkcolor = black, citecolor = antiquefuchsia]{hyperref}

\makeatother

\numberwithin{equation}{section}

\newtheorem{theorem}{Theorem}[section]
\newtheorem{lemma}[theorem]{Lemma}
\newtheorem{proposition}[theorem]{Proposition}

\newtheorem{definition}[theorem]{Definition}
\theoremstyle{remark}
\newtheorem{remark}[theorem]{Remark}

\newcommand{\N}{\mathbb{N}}
\newcommand{\Z}{\mathbb{Z}}
\newcommand{\R}{\mathbb{R}}

\newcommand{\mres}{\mathbin{\vrule height 1.6ex depth 0pt width
0.13ex\vrule height 0.13ex depth 0pt width 1.3ex}}

\DeclareMathOperator{\dist}{dist}
\DeclareMathOperator{\Per}{Per}

\renewcommand{\d}{{\rm d}}
\newcommand{\HH}{{\mathcal H}}

\def\Z{\ensuremath\mathbb{Z}}

\usepackage[capitalise,nameinlink]{cleveref}
\crefformat{equation}{#2(#1)#3}
\crefrangeformat{equation}{#3(#1)#4 to~#5(#2)#6}
\crefmultiformat{equation}{#2(#1)#3}{ and~#2(#1)#3}{, #2(#1)#3}{ and~#2(#1)#3}

\title[]{Lattice tilings with minimal perimeter\\ and unequal volumes}
\author[]{Francesco Nobili} 
\address{Universit\'a di Pisa, Dipartimento di Matematica, Largo Bruno Pontecorvo 5,
56127 Pisa, Italy}
\email{\url{francesco.nobili@dm.unipi.it}}
\author[]{Matteo Novaga} 
\address{Universit\'a di Pisa, Dipartimento di Matematica, Largo Bruno Pontecorvo 5,
56127 Pisa, Italy}
\email{\url{matteo.novaga@unipi.it}}

\date{\today. 
}

\begin{document}
\begin{abstract}
We study periodic tessellations of the Euclidean space with unequal cells arising from the minimization of perimeter functionals. Existence results and qualitative properties of minimizers are discussed for different classes of problems, involving local and non-local perimeters. Regularity is then addressed in the general case under volume penalization, and in the planar case with the standard perimeter, prescribing the volumes of each cell. Finally, we show the optimality of hexagonal tilings among partitions with almost equal areas.
\end{abstract}

\maketitle
\tableofcontents

\section{Introduction} 
A relevant problem in geometry, proposed by Lord Kelvin in 1887 \cite{Kelvin1887}, consists in finding the partition of $\R^d$ into cells of equal volume, so that the area of the surfaces separating them is as small as possible. In two dimensions, a complete solution was achieved by T. C. Hales in \cite{Hales01} who proved that the minimal configuration is the regular hexagonal tiling, thus solving the celebrated Honeycomb conjecture (see also \cite{Fejes1,Fejer2,Fejer3,Fejer4,MorganChristopherGreenleaf1988,Morgan1999}). 
In three dimensions Lord Kelvin proposed a possible candidate minimizer, the so-called Kelvin foam, 
but an example with lower (average) surface area was later provided by Weaire and Phelan in \cite{WeairePhelan1996},
so that the original problem remains open, and the same holds in higher dimensions.

These types of questions naturally arise in the calculus of variation, when tilings are recast as minimizers of energy functionals involving perimeter measures under partitioning constraints. In particular, the theory of functions of bounded variation (see, e.g., \cite{AmbrosioFuscoPallarabook}) and, more specifically, that of sets finite perimeter (see, e.g., \cite{Maggi12_Book, MorganBOOK}) are a natural language to adopt here, with a powerful regularity theory at disposal to address these types of problems. 

\medskip

In this work, we consider partitions that are invariant by the action of a $d$-dimensional group of translations in $\R^d$, i.e., a \emph{lattice}. We also admit partitions of $\R^d$ with unequal cells. More precisely, we will look for minimizers of general perimeter functionals in the class of fundamental domains for the action of the group. The approach was previously adopted in \cite{Choe89} for $d=3$ and for the classical perimeter, and later also in \cite{MartelliNovagaPludaRiolo2017,NovagaPaoliniStepanovTortorelli22,NovagaPaoliniStepanovTortorelli23,CesaroniNovaga23_1,CesaroniFragalaNovaga23}. 

We start by fixing the notation and discuss our main results.
Given a lattice $G$ (i.e.\ a discrete group of $\R^d$) with volume $d(G)>0$, a fundamental domain for the action of $G$ is a set $D\subseteq \R^d$ so that $|D|=d(G)$ and $ (D+g)$ covers $\R^d$ with no overlapping, up to negligible sets, when $g$ runs among all elements of the group $G$ (see Section \ref{sec:lat} for details).
We then consider two notions of perimeter:
\[
\Per_\varphi(E,\cdot),\quad\text{and}\quad \Per_K(E,\cdot), \quad\text{for all } E\subseteq \R^d\text{ Borel},
\]
the local anisotropic perimeter, with $G$-invariant anisotropic norm $\varphi \colon \R^d \to [0,\infty)$, and the non-local perimeter with respect to a non-negative interaction kernel $K \colon \R^d \to [0,\infty)$ that is singular at the origin with a fractional type singularity (see  Section \ref{sec:perimeters} for precise definitions and relevant properties). 

The first problem we tackle is the existence of minimal tessellations of $\R^d$, with possibly unequal cells,
which are periodic with respect to a given lattice $G$. This can be achieved by prescribing a volume vector $\vec v = (v_1,...,v_N)$, 
where $v_i>0$ for all $i=1,...,N$ and $\sum_{i=1}^N v_i = d(G)$ for some $N\in\N$, and then by looking at minimizers of 
\begin{equation}\label{main:constrained G} 
    \inf_D \inf \left\{ \frac 12\sum_i\Per(E_i) \colon \begin{array}{ll}
        E_i\subseteq D, &|E_i|=v_i \\
         |E_i\cap E_j|=0, &\big|D\setminus \cup_i E_i\big|=0, \end{array}\right\},
\end{equation}
where the first infimum runs over all fundamental domains $D$ for the action of $G$ and $\Per(\cdot)$ denotes either the local or the 
non-local perimeter. We will prove existence of minimizers for the problem above in Theorem \ref{thm:step 2}.

It is then natural to perform a further minimization step by letting the lattice to vary:
\begin{equation}\label{main:constrained}
    \inf_G\inf_D \inf \left\{ \frac 12\sum_i\Per(E_i) \colon \begin{array}{ll}
        E_i\subseteq D, &|E_i|=v_i \\
         |E_i\cap E_j|=0, &\big|D\setminus \cup_i E_i\big|=0, \end{array}\right\},
\end{equation}
where the first infimum now runs among all lattices $G$ with prescribed volume $d(G)=L^d$, for some $L>0$. 
Existence of minimizers will be proved in Theorem \ref{thm:step 3}.

We can also relax the volume constraints and consider, for some penalization parameter $\lambda>0$, the minimum problem
\begin{equation}\label{main:penalized}
\inf_G\inf_D \inf \left\{  \frac 12\sum_i \Per(E_i) + \lambda \big| |E_i|-v_i\big|  \colon \begin{array}{l}
        E_i\subseteq D \\ |E_i\cap E_j|=0 \\
        \big|D\setminus \cup_i E_i\big|=0
    \end{array}  \right\},
\end{equation}
where we do not require $|E_i|=v_i$, since it's encoded in the functional to be minimized. The existence of minimizers will be proved in Theorem \ref{thm:penalized}.

In all these problems, existence of minimal configurations will be proved by the direct method of the calculus of variations, 
compactness will be obtained by translating back countably many pieces of almost minimizers that are possibly escaping to infinity, 
following the method in \cite{NovagaPaoliniStepanovTortorelli22,CesaroniNovaga22,CesaroniNovaga23_1}.

\medskip

Given minimizers $(E_i),D,G$ of any of the above problems, we have that $\{ E_i +g \colon i, g\in G\}$
is a partition of $\R^d$, up to negliglible sets.
In all cases, the above minimal partition comes with a \emph{local minimality} property on sufficiently small balls, depending on the packing radius of $G$. Here, local minimality is understood in the sense of sets of finite perimeter for compactly supported perturbations, see Section \ref{sec:minimality} for the precise definitions. In particular, local minimality typically gives access to regularity theory. Here, the main challenge is to understand the local finiteness of the partition to invoke classical regularity results for isoperimetric problems on clusters. A key difference in our problems is that the partitions associated with a minimizer of \eqref{main:penalized} are much well-behaved compared to those in \eqref{main:constrained G}, \eqref{main:constrained}, as they do not require to fix volumes of the compactly supported perturbations. Compare Lemma \ref{lem:minimality constrained} with Lemma \ref{lem:minimality penalized} that in turn makes it possible to derive regularity results in Proposition \ref{prop:regularty penalized 1}, Proposition \ref{prop:regularty penalized 2} and Proposition \ref{prop:regularity penalized 3}. 
In this respect, it would be interesting to better understand the regularity of the volume-constrained problems. For instance, is to be understood if minimizers for the penalized problem \eqref{main:penalized} are also minimizers for the constrained one \eqref{main:constrained G}, when $\lambda\gg 1$ is sufficiently big. Moreover, the dependence of the regularity on the total volume of the lattices $L^d$ a priori depends on $L$. This is an obstacle to obtaining $L^1_{loc}$-compactness of minimal partitions when letting $L\uparrow\infty$, so to reach the non-periodic scenario. However, in this case, we expect volume constraints to be in general lost in the limit (see also a related discussion \cite{Morgan1999}).

\medskip

When restricting to the planar case, we have access to more tools in order to understand the regularity of minimal partitions. In this case, we are able to prove that minimal partitions are locally finite 
and therefore reduce to the classical theory of cluster with standard Euclidean perimeter, see Theorem \ref{thm:regularity constrained planar}.

Given that we are here interested in partitions with unequal volumes, it is natural to ask if the hexagonal partition is stable among periodic partitions with \emph{almost equal volumes}. We address this issue in Theorem \ref{thm:stability HC}, by relying on the stability of networks \cite{pludaPozzetta23} (see also \cite{FisherHenselLauxSimon2023}) and a contradiction argument involving fine decomposition of planar sets of finite perimeter (see \cite{AmbrosioCasellesMasnouMorel01}), and improved convergence of \emph{good parts} of the boundaries of a minimizing sequence. In particular, we show that small modifications of the Honeycomb partition (see Figures \ref{fig:pic3} and \ref{fig: D with H v}) minimize \eqref{main:constrained} in the volume regime $v_i\approx 1$ for all $i$. 

We finally observe that the same result holds for the anisotropic perimeter when the corresponding Wulff shape is a regular hexagon.

\medskip

\noindent\textbf{Acknowledgments}. The authors are members of INDAM-GNAMPA, and acknowledge partial support by the MIUR Excellence Department Project awarded to the Department of Mathematics, University of Pisa, and by the PRIN Project 2022 GEPSO. Both authors would like to thank A. Pluda for useful discussions.

\section{Notation and preliminary definitions}

\subsection{Lattices}\label{sec:lat}
Given $d$-vectors $(e_i)_{i=1}^d\subseteq \R^d$, we denote $\det(e_i)$ the determinant of the matrix whose $i$-column is $e_i$. We recall the notion of lattice.

\begin{definition}[Lattice]
    A \emph{lattice} is a discrete subgroup $G$ of $(\R^d,+)$ of rank $d$. Any $g \in G$ is uniquely determined by coefficients $(k_i)_{i=1}^d\subseteq \Z^d$ and a given basis $(e_i)_{i=1}^d\subseteq \R^d$ via $g = \sum_{i=1}^d k_ie_i.$ The \emph{volume} of the lattice $G$ is the number $d(G)\coloneqq |\det(e_i)|$.
\end{definition}
We observe that the volume of a lattice $G$ is invariant by the choice of the basis $(e_i)$ of the lattice. Indeed, any two bases are related by an orthonormal transformation with determinant $\pm 1$. Let us also define the packing and covering radius of $G$, respectively given by
\begin{align*}
\rho_G&\coloneqq\sup \{ r>0 \colon B_r(g_1)\cap B_r(g_2) = \emptyset \text{ for all }g_1\neq g_2 \in G\};\\
r_G&\coloneqq \inf\{r>0 \colon \cup_g (B_r(0)+g) = \R^d \}.
\end{align*}

\begin{definition}[Fundamental Domain]
A \emph{fundamental domain} for the lattice $G$ is a Borel set $D\subseteq \R^d$ so that $|(D+g)\cap D |=0$ for every $g\in G,\, g\neq {\rm id}$ and $|\R^d \setminus \bigcup_{g\in G}(D+g)|=0$.
\end{definition}
Given a lattice $G$ and a basis $(e_i)$, we can consider 
\begin{equation}
  D_G \coloneqq \Big\{ \sum_i t_i e_i \colon t_i \in [0,1), i=1,..,n\Big\}\subseteq \R^d.
\label{eq:parallelepiped}  
\end{equation}
In what follows we shall use that $D$ is a fundamental domain for $G$.
\begin{remark}\label{eq:volume fund domain}
   Let us fix a lattice $G$ with $d(G)=V$, for some $V>0$. We remark that any fundamental domain $D$ for the lattice $G$ satisfies $|D|=V$. Indeed, letting $(e_i)$ be a basis of $G$ and $D_G$ be defined as above,
   since $D_G$ is a fundamental domain for $G$, for any other fundamental domain $D$ we have
\begin{equation}
V = |D_G| = \Big|\bigcup_{g \in G} D_G \cap (D+g) \Big| = \Big|\bigcup_{g \in G} (D_G -g) \cap D\Big| =|D|.
\end{equation}
\end{remark}

\subsection{Perimeters}\label{sec:perimeters} In this note, we shall consider two different types of Perimeter measures on $\R^d$, namely anisotropic local perimeters and non-local perimeters. We briefly introduce here the definitions and list afterwards the properties we are going to use.

For every $\Omega\subseteq \R^d$ open and $E\subseteq\R^d$ Borel, we consider the anisotropic perimeter
    \[
    \Per_\varphi (E,\Omega)\coloneqq \int_{\partial^*E\cap\Omega}\varphi (\nu_E(x))\,\d \HH^{d-1}(x),
    \]
    where $\varphi$ is a norm in $\R^d$. Here, $E$ is a set of finite perimeter,  $\partial^*E, \nu_E$, are respectively the reduced boundary of $E$ and the unit outer normal, and $\HH^{d-1}$ is the $d-1$ Hausdorff measure, see \cite{AmbrosioFuscoPallarabook} for the classical theory and \cite{Maggi12_Book} for the anisotropic case.
 
We will also consider non-local perimeters with singularity of fractional type at the origin following \cite[Section 6]{NovagaPaoliniStepanovTortorelli22},\cite[Section 1.2]{CesaroniNovaga22} and references therein. Let $K \colon \R^d \to \R$ be an interaction Kernel satisfying:
\begin{itemize}
    \item[a)] $K(x)= K(-x)$ for all $x\in\R^d$;
    \item[b)] $\min( 1, |x|)K(x) \in L^1(\R^d)$;
    \item[c)] there is $C>0$ and $s \in (0,1)$ so that $K(x)>C |x|^{-d-s}$.
\end{itemize}
We can then define the non-local perimeter of a Borel set $E\subseteq \R^d$ and the non-local relative perimeter of $E$ on an open set $\Omega\subseteq \R^d$, respectively as
    \begin{align*}
       \Per_K(E) &\coloneqq \int_E\int_{\R^d\setminus E}K(x-y)\,\d x \d y \\
       \Per_K(E,\Omega) &\coloneqq \int_{E\cap \Omega}\int_{\R^d\setminus E}K(x-y)\,\d x \d y + \int_{E\setminus \Omega}\int_{\Omega\setminus E}K(x-y)\,\d x \d y.
    \end{align*}

We now list some key properties shared by the above perimeter measures. To be more concise, we unify the presentation and we consider $\Per(\cdot,\cdot)$ to be either the local or non-local perimeter in the following:
    \begin{itemize}
        \item[]{\sc Semicontinuity}. If $E_n \to E$ in $L^1_{loc}$, then $\Per(E,\Omega)\le \liminf_{n\uparrow\infty}\Per(E_n,\Omega)$;
        \item[]{\sc Monotonnicity}. If $\Omega\subseteq \Omega'$, then $\Per(E,\Omega)\le \Per(E,\Omega')$; moreover, if $\Omega_i\subseteq \Omega_{i+1}$, then $\Per(E,\Omega_i)\to\Per(E,\cup_i \Omega_i)$;
        \item[]{\sc Compactness}. If $(E_n)$ are so that $\sup_n\Per(E_n,\Omega)<\infty$ for some pre-compact open set $\Omega$, then up to subsequence $E_n\cap \Omega \to E\cap \Omega$ for some Borel set $E$;
        \item[]{\sc Almost Subadditivity}. if $\Omega_i$ are pairwise disjoint and open, we have:
        \begin{itemize}
            \item[{\rm i})] it holds
                \begin{align*}
                    \Per(E,\cup_i\Omega_i) &\le \sum_i \Per(E,\Omega_i)\le \Per(E,\cup_i \Omega_i) \\
                    &\qquad + \sum_i |E\cap \Omega_i|\min_{j\neq i} \phi({\rm dist}(\Omega_i,\Omega_j));
                \end{align*}
            \item[{\rm ii)}] there is $c\ge 1$ so that, if $|\R^d\setminus \cup_i\Omega_i|=0$, then $\sum_i\Per(E,\Omega_i) \le c\Per(E)$;
        \end{itemize}
       in the case of the anisotropic perimeter we have $\phi =0,c=1$, while in the case of the non-local perimeter, we have $c=2,\phi(t) =\int_{\R^d\setminus B_t(0)}K(x)\,\d x$.
    \end{itemize}
    Moreover, we relate the Perimeter measure with the lattice $G$ by assuming the following $G$-invariance property:
    \begin{itemize}
        \item[]{\sc $G$-invariance}. It holds that $\Per((E+g),(\Omega +g)) = \Per(E,\Omega)$ for every $g \in G$.
    \end{itemize}
    Notice that in the anisotropic case, this is implied by the translation invariance of the norm $\varphi$.
    
\section{Existence results}
\subsection{Minimizing periodic partitions}
Given a lattice and a fundamental domain for its action, we now prove existence of a minimizing partition with volume constraints.
\begin{proposition}\label{prop:step 1}
    Let $L>0,\mu \ge 0$, let $G$ be a lattice in $\R^d$ with $d(G)=L^d$ and let $D\subseteq \R^d$ be a fundamental domain for $G$. Fix $N \in \N$ and a volume vector $ \vec v \coloneqq (v_1,...,v_N)$ so that $v_i> 0, \sum_i v_i = L^d$. Consider the minimization problem
    \[
    f_\mu(D,\vec v)\coloneqq  \inf \left\{ \mu \Per(D) + \frac12\sum_i\Per(E_i) \colon \begin{array}{ll}
        E_i\subseteq D\emph{ Borel}, &|E_i|=v_i \\
         |E_i\cap E_j|=0, &\big|D\setminus \cup_i E_i\big|=0
    \end{array}  \right\},
    \]
    where $\Per$ is either the local or non-local perimeter.

    If $f_\mu(D,\vec v)$ is finite, then there exists a minimizing partition $(E_i)_i$ of $D$.
\end{proposition}
\begin{proof}
Let us write $f_\mu(D)$ for short in this proof dropping the dependence on $\vec v$ which is going to be considered fixed. Also, notice that the arguments would simplify in the case  $\mu=0$.

Since $f_\mu(D)$ is assumed finite, we can consider a minimizing sequence $(E_i^k)$ so that
\[
f_\mu(D) = \lim_{k\uparrow\infty} \mu \Per(D) + \frac12\sum_i\Per(E^k_i).
\]
Possibly passing to a non-relabeled subsequence in $k$, we can assume further that
\[
\sup_{k\in\N} \Per(E_i^k) <\infty,\qquad \forall i.
\]
Recalling that $|E^k_i|=v_i$ for all $k\in\N$, we have by the compactness of sets of finite perimeter (both, in the local and non-local case) and a diagonalization argument the existence of a common non-relabeled subsequence in $k$ so that
\[
E_i^k \to E_i,\qquad \text{in }L^1_{loc}\text{ and for all }i,
\]
as $k\uparrow\infty$, for some Borel sets $E_i$. The rest of the proof is devoted to showing that the partition $(E_i)$ is admissible for $f_\mu(D)$ and it holds
\[
f_\mu(D) = \mu \Per(D)+\frac12 \sum_i\Per(E_i).
\]
The fact that $|E_i^k\cap E^k_j|=0,|D\setminus \bigcup_i E^k_i|=0$ for all $k$ and lower semicontinuity gives that $|E_i\cap E_j|=0,|D\setminus \bigcup_i E_i|=0$, for $i\neq j$ and that  $|E_i|\le v_i$ for all $i$. We claim that actually $|E_i| = v_i$. First of all, notice, by Remark \ref{eq:volume fund domain} and the fact that all $E^k_i$ are disjoint up to negligible sets, that
\[
L^d = |D| = \sum_i |E_i^k| = \sum_i v_i ,\qquad\forall k \in \N.
\]
Thus, combining $|D|=L^d, |D\setminus \bigcup_i E_i|=0$ with $|E_i|\le v_i$, we get the claim. Moreover, possibly intersecting each $E_i$ with $D$ discarding a negligible set by what we just proved, we can guarantee that $E_i\subseteq D$. We thus showed that the partition $(E_i)$ is admissible for $f_\mu(D)$. 

Finally, we prove that $f_\mu(D)$ is attained by $(E_i)$. This easily follows noticing that
\[
f_\mu(D) = \mu \Per(D) + \frac 12 \lim_{k\uparrow\infty}   \sum_i\Per(E_i^k) \ge\mu \Per(D) +\frac 12 \sum_i\Per(E_i),
\]
by lower semicontinuity of the perimeter measure with respect $L^1_{loc}$ convergence.
\end{proof}
Next, we use the above proposition to prove the following existence result. Notice that, combined with the above, the following produces a $G$-periodic partition of $\R^d$ for any given lattice $G$.
\begin{theorem}\label{thm:step 2}
    Let $L>0,\mu \ge 0$ and let $G$ be a lattice in $\R^d$ with $d(G)=L^d$.  Fix $N \in \N$ and a volume vector $ \vec v \coloneqq (v_1,...,v_N)$ so that $v_i>0, \sum_i v_i = L^d$.  Consider the minimization problem
    \begin{align*}
        F_\mu (G,\vec v)&\coloneqq \inf \{f_\mu (D,\vec v)\colon D\subseteq \R^d \emph{ fundamental domain for } G\} \\
        &=\inf_D \inf \left\{ \mu \Per(D) + \frac12\sum_i\Per(E_i) \colon \begin{array}{ll}
        E_i\subseteq D\emph{ Borel}, &|E_i|=v_i \\
         |E_i\cap E_j|=0, &\big|D\setminus \cup_i E_i\big|=0
    \end{array}  \right\},
    \end{align*}
    where $\Per$ is either the local or non-local perimeter. Then, there exists a minimal fundamental domain $D$. 
\end{theorem}
\begin{proof}
    As before, we write $ F_\mu (G),f_\mu(D)$ for short considering $\vec v$ fixed and we point out that the case $\mu=0$ would simplify the following arguments.

    First, we show that $F_\mu(G)$ is finite. Let us consider the fundamental domain $D_G$ as a competitor in the definition of $F_\mu(D)$ (i.e.\ the parallelepiped \eqref{eq:parallelepiped}). We will now construct iteratively an admissible partition $(E_i)$ of $D_G$. Observe that the mapping $\R \ni t \mapsto v(t)\coloneqq |D_G\cap \{ x_n\le t\}|$ is continuous, monotone increasing, non-negative and bounded above by $d(G)=L^d$. Thus, we define iteratively
    \begin{align*}
       & E_1 \coloneqq D_G\cap \{ x_n\le t_1\} && \text{for $t_1 \in \R$ so that }v(t_1) =v_1,\\
       &E_i \coloneqq D_G\cap \{ t_{i-1}<x_n\le t_i\}&& \text{for $t_i \in \R$ so that }v(t_i)-\sum_{j<i}v_j = v_i,
    \end{align*}
    for all $i=2,...,N$. By construction, we clearly have $E_i\subseteq D_G$, $E_i\cap E_j =\emptyset$ and $|E_i|=v_i$ as well as $D_G=\cup_i E_i$. Hence, the partition $(E_i)$ is admissible and satisfies
    \begin{equation}
    F_\mu(G)\le f_\mu(D_G) \le \mu \Per(D_G) + \frac{1}{2}\sum_i\Per(E_i) \le \Big(\mu + \frac{N}{2}\Big)\Per(D_G) <\infty,\label{eq:C mu N}
    \end{equation}
    as desired.
    
    We can now consider a minimizing sequence $(D^k)$  of fundamental domains for the action of $G$ so that
    \[
    f_\mu(D^k)\to F_\mu(G),\qquad\text{as }k\uparrow\infty.
    \]
    In particular, up to a not relabeled subsequence, we might suppose that $f_\mu(D^k)$ is finite for all $k\in\N$.  We can thus apply Proposition \ref{prop:step 1} to deduce the existence of a partition $(E_i^k)$ so that
    \[
    f(D^k) = \mu \Per(D^k)+\frac 12\sum_i\Per(E^k_i),
    \]
    satisfying all the listed properties in the definition of $f_\mu(\cdot)$ and, by what previously said, also
    \begin{equation}
        L^d = |D^k|,\qquad \sup_{k\in\N}\Per(D^k) <\infty,\qquad \sup_{k\in\N}\Per( E^k_i) <\infty,\quad \forall i\label{eq:bounds}
    \end{equation}

    We are in the position to invoke the concentration compactness result \cite[Lemma 3.4]{CesaroniNovaga22} (observed that both the local and non-local perimeter treated here are covered, cf. \cite[Section 2.1]{CesaroniNovaga22}) to deduce, up to a further non relabeled subsequence in $k$, the existence of $(g_l^k)\subseteq G$ with $g_l^k\neq g_m^k$ for $m\neq l$ and all $k\in\N$ and sets $(D_l)\subseteq \R^d$ satisfying 
    \[
    \sum_l|D_l|=L^d,\qquad (D^k -g_l^k) \to D_l,\quad\text{ in $L^1_{loc}$ as }k\uparrow\infty \text{ and for all }l.
    \]
    From \cite[Lemma 3.4]{CesaroniNovaga22}  we also have the property that ${\rm dist}(K+g^k_l,K +g_m^k) \to \infty$ for each $K\subseteq \R^d$ compact and $m\neq l$. By a diagonalization argument and precompactness of sets of finite perimeters guaranteed  by \eqref{eq:bounds} and $|E^k_i|=v_i$ for all $k$, we can pass to a common further subsequence in $k$ so that
    \[
    (E^k_i - g_l^k ) \to E_{i,l},\qquad\text{ in $L^1_{loc}$ as }k\uparrow\infty\text{ and for all }i,l,
    \]
    for some Borel sets $(E_{i,l})\subseteq \R^d$. Set
    \[
    E_i \coloneqq \bigcup_l E_{i,l},\quad \text{for all } i,\qquad D\coloneqq \bigcup_l D_l. 
    \]
    Clearly, up to discard negligible sets, we can suppose $E_{i,l}\subseteq D_l$. We now  make three claims:
    \begin{itemize}
        \item[{\rm i)}] $D$ is a fundamental domain for the action of $G$;
        \item[{\rm ii)}] $|E_i|=v_i$ for all $i$;
        \item[{\rm iii)}] $|E_i \cap E_j|=0$ for $i\neq j$ and $|D\setminus \bigcup_iE_i|=0$.
    \end{itemize}
    We first show i). Thanks to the fact that $g_l^k\neq g_m^k$ for $l\neq m$ and for all $k$, we have that $|D_l\cap D_m| =0$. This is a simple consequence of the fact that $|(D^k -g_l^k) \cap (D^k -g_m^k) | =0$ for every $k$, using that $D^k$ is a fundamental domain and the lower semicontinuity of volumes.  Moreover, again by lower semicontinuity, it holds for any $g\in G,g \neq {\rm id}$ that $|D \cap (D+g) | \le \liminf_{k\uparrow\infty}|D^k \cap (D^k+g)|=0$ since $D^k$ is a fundamental domain and similarly, using that $|D^k|= L^d= \sum_l|D_l|$, it holds $|\R^d \setminus \bigcup_{g\in G}(D+g)| =0$. This concludes the proof of i).

    We now show iii). For every $i\neq j$ we have 
    \begin{align*}
        | E_i \cap E_j|&\le \sum_{l,m} |E_{i,l}\cap E_{j,m}| \le\liminf_{k\uparrow\infty}  \sum_{l,m}  \big|  (E^k_i - g_l^k )\cap  (E^k_j - g_m^k )\big|\\
        &=\liminf_{k\uparrow\infty}  \sum_{l\neq m}  \big|  (E^k_i - g_l^k )\cap  (E^k_j - g_m^k )\big|\\
        &\le\liminf_{k\uparrow\infty}  \sum_{l\neq m} |(D^k -g_l^k) \cap (D^k -g_m^k) | =0,
    \end{align*}
    having used: in the second inequality the lower semicontinuity of volumes, Fatou's lemma and the fact that the intersection of $L^1_{loc}$-converging sets is again $L^1_{loc}$ converging; in the middle identity that, since $(E^k_i)$ is a partition of $D^k$, then $ \big|  (E^k_i - g_l^k )\cap  (E^k_j - g_l^k )\big|=0,$ when $i\neq j$ for all $k$; in the last inequality a simple set inclusion; finally, the last sum vanishes, as $D^k$ is a fundamental domain and $g_l^k \neq g_m^k$. By similar lower semicontinuity arguments using that $|D^k\setminus \cup_i E^k_i|=0$ we complete the proof of iii).
    
    We now prove ii). For every $i$, we can estimate thanks to the above
    \begin{align*}
         |E_i| &\le \sum_l |E_{i,l}|  = \sum_l |E_{i,l}\cap D|  \le \liminf_{k\uparrow\infty}\sum_l|(E_i^k -g^k_l)\cap D)| \le \liminf_{k\uparrow\infty}\sum_l|E_i^k \cap (D +g_l^k)|\\
         &\le \liminf_{k\uparrow\infty}\sum_{g \in G}|E_i^k \cap (D +g)| = \liminf_{k\uparrow\infty}|E_i^k|=v_i,       
    \end{align*}
    having also used that $E_{i,l}\subseteq D$, lower semicontinuity and that $D$ is fundamental. In particular, recalling by iii) that $|D \setminus \cup_i E_i|=0$ and $|D|=L^d$, we necessarily deduce that $|E_i|=v_i$ for all $i$.

    All in all, we found a partition $(E_i)$ of the fundamental domain $D$ which is admissible for $f_\mu(\cdot)$.  To conclude, we only need to show that $F_\mu(G)$ is attained by $D$. Finally, this follows by the lower semicontinuity result in Lemma \ref{lem:semicontinuity} below, applied several times with $G^k\equiv G$ for $D^k$ and for $E_i^k$ for each $i$.
\end{proof}
\subsection{Minimizing lattice tilings}
\begin{theorem}\label{thm:step 3}
    Let $L>0,\mu\ge 0$.  Fix $N \in \N$ and a volume vector $ \vec v \coloneqq (v_1,...,v_N)$ so that $v_i>0, \sum_i v_i = L^d$.  Consider the minimization problem
    \begin{align*}
         F_* &\coloneqq \inf \{F_\mu(G,\vec v) \colon G \text{ lattice with }d(G)=L^d \}\\
        &=\inf_G \inf_D \inf \left\{ \mu \Per(D) + \frac12\sum_i\Per(E_i) \colon \begin{array}{ll}
        E_i\subseteq D\emph{ Borel}, &|E_i|=v_i \\
         |E_i\cap E_j|=0, &\big|D\setminus \cup_i E_i\big|=0
    \end{array}  \right\},
    \end{align*}
    where $\Per$ is either the anisotropic local or non-local perimeter.

    Then, there exist a minimizer lattice $G$, a fundamental domain $D$ for the action of $G$ minimizing $F_\mu(G,\vec v)$ and a partition of $(E_i)$ minimizing $f_\mu(D,\vec v)$, hence satisfying
    \[
    F_* = \mu \Per(D) +\frac 12 \sum_i \Per(E_i).
    \]
\end{theorem}
\begin{proof}
     As before, we write $ F_\mu (G),f_\mu(D)$ for short considering $\vec v$ fixed.

     We consider a minimizing sequence of lattices $(G^k)$ so that $F_\mu(G_k)\to F_*$ as $k\uparrow\infty$. By Theorem \ref{thm:step 2}, we infer the existence of a fundamental domain minimizer $D$ so that $F_\mu(G^k)=f_\mu(D^k)$. By Proposition \ref{prop:step 1}, there exists then a partition $(E^k_i)$ of $D$ satisfying
    \[
    \mu\Per(D^k)+\frac 12\sum_i\Per(E_i^k)\to F_*,\qquad \text{as }k\uparrow\infty.
    \]
    In particular, up to a non-relabeled subsequence, we have the following uniform bounds
    \begin{equation}
        L^d = d(G^k)=|D^k|,\qquad \sup_{k\in\N}\Per(D^k) <\infty,\qquad \sup_{k\in\N}\Per( E^k_i) <\infty,\quad \forall i\label{eq:bounds lattices}
    \end{equation}
    The strategy is now similar to that already employed in the proof of  Theorem \ref{thm:step 2}. We repeat the arguments to handle the situation where $G^k$ also varies.

    We are in the position to invoke the concentration compactness result \cite[Lemma 3.1]{CesaroniFragalaNovaga23} to deduce, up to a non-relabeled subsequence, the existence of a lattice $G$ with $d(G)=L^d$, of $(g_l^k)_l \subseteq G^k$ and of Borel sets $D_l\subseteq \R^d$ so that,
    \[
     G_k \to G,\qquad \text{in the sense of Kuratowski},
    \]
    and
    \begin{align*}
        \sum_l|D_l|= L^d,\qquad D^k-g_l^k \to D_l\quad\text{in }L^1_{loc}\text{ as $k\uparrow\infty$ for all $l$}. 
    \end{align*}
    Furthermore (by \cite[Lemma 3.1]{CesaroniFragalaNovaga23} again), we have that $|D_l\cap D_m|=0$ and $D\coloneqq \cup_lD_l$ is a fundamental domain for the limit lattice $G$ and, for every $K\subseteq \R^d$ compact, it holds  ${\rm dist}(K+g^k_l,K +g_m^k) \to \infty$ for all $m\neq l$ as $k\uparrow\infty$. Actually, the above conclusions are stated in \cite{CesaroniFragalaNovaga23} for the non-local perimeter only. However, the same conclusion holds in the case of anisotropic local perimeter by \cite[Lemma 3.2,]{CesaroniNovaga23_1} combined with the arguments already developed in \cite[Theorem 3.3]{CesaroniNovaga23_1}. We can thus proceed with the proof.

    By a diagonalization argument and precompactness of sets of finite perimeters guaranteed  by \eqref{eq:bounds lattices} and $|E^k_i|=v_i$, we can pass to a common further subsequence in $k$ so that
    \[
    (E^k_i - g_l^k ) \to E_{i,l},\qquad\text{ in $L^1_{loc}$ as }k\uparrow\infty\text{ and for all }i,l,
    \]
    for some Borel sets $(E_{i,l})\subseteq \R^d$. Set now
    \[
     E_i \coloneqq \bigcup_l E_{i,l},\quad \forall i.
    \]
    We now claim the following:
    \begin{itemize}
        \item[{\rm i)}] $|E_i|=v_i$ for all $i$;
        \item[{\rm ii)}] $|E_i \cap E_j|=0$ for $i\neq j$ and $|D\setminus \bigcup_iE_i|=0$.
    \end{itemize}
    Let us show ii) first. For every $i\neq j$ we have 
    \begin{align*}
        | E_i \cap E_j|&\le \sum_{l,m} |E_{i,l}\cap E_{j,m}| \le\liminf_{k\uparrow\infty}  \sum_{l,m}  |  (E^k_i - g_l^k )\cap  (E^k_j - g_m^k )|\\
        &=\liminf_{k\uparrow\infty}  \sum_{l\neq m}  |  (E^k_i - g_l^k )\cap  (E^k_j - g_m^k )|\\
        &\le\liminf_{k\uparrow\infty}  \sum_{l\neq m} |(D^k -g_l^k) \cap (D^k -g_m^k) | =0,
    \end{align*}
    having used: in the second inequality the lower semicontinuity of volumes, Fatou's lemma and the fact that the intersection of $L^1_{loc}$-converging sets is again $L^1_{loc}$ converging; in the middle identity that, since $(E^k_i)$ is a partition of $D^k$, then $ |  (E^k_i - g_l^k )\cap  (E^k_j - g_l^k )|=0,$ when $i\neq j$ for all $k$; in the last inequality a simple set inclusion; finally, the last sum vanishes, as $D^k$ is a fundamental domain and $g_l^k \neq g_m^k$. By similar arguments, using that $|D^k\setminus \cup_i E^k_i|=0$, ii) follows.

    We now prove i). For every $i$, we can estimate thanks to the above
    \begin{align*}
         |E_i| &\le \sum_l |E_{i,l}|  = \sum_l |E_{i,l}\cap D|  \le \liminf_{k\uparrow\infty}\sum_l|(E_i^k -g^k_l)\cap D)| \\
         &\le \liminf_{k\uparrow\infty}\sum_l|E_i^k \cap (D +g_l^k)|\le \liminf_{k\uparrow\infty}\sum_{g \in G}|E_i^k \cap (D +g)| = \liminf_{k\uparrow\infty}|E_i^k|=v_i,       
    \end{align*}
    having also used that $E_{i,l}\subseteq D$, lower semicontinuity and that $D$ is fundamental. In particular, recalling by ii) that $|D \setminus \cup_i E_i|=0$ and $|D|=L^d$, we necessarily deduce that $|E_i|=v_i$ for all $i$.
    
    We found a partition $(E_i)$ of the fundamental domain $D$ for the lattice $G$ which is admissible for $f_\mu(\cdot)$.  We next show that $F_*$ is attained by $G$. This follows by the lower semicontinuity result in Lemma \ref{lem:semicontinuity} below  applied several times for $D^k$ and for $E_i^k$ for each $i$ giving in turn
    \[
    \mu \Per(D)+\frac 12 \sum_i \Per(E_i)\le \liminf_kf_\mu(D^k) = \lim_kF_\mu(G^k) = F_*.
    \]
\end{proof}
\begin{lemma}\label{lem:semicontinuity}
    Let $(G^k)$ be a sequence of lattices in $\R^d$ and $(E^k)$ a sequence of Borel sets. Suppose there are $(g^k_l)_l\subseteq G^k$ so that $\dist(K+ g^k_l,K+g^k_m)\to \infty$ as $k\uparrow\infty$ for every $K\subseteq \R^d$ compact, for all $l\neq m$. Suppose there are Borel sets $E_l$ so that
    \[
        E^k-g_l^k \to E_l,\qquad \text{in $L^1_{loc}$ as $k\uparrow\infty$ and for all  $l$.}
    \]
    Then, if $\Per$ is either the local or non-local perimeter, it holds
    \[
     \Per(\cup_l E_l)\le \liminf_{k\uparrow\infty} \Per(E^k).
    \]
\end{lemma}
\begin{proof}
    If the right-handmost side is infinite, there is nothing to prove. Suppose it is then finite. Let $B \coloneqq B_R(0)$ for $R>0$ and for every $M,k\in\N$, we can suppose therefore the existence of a $k_0\coloneqq k_0(M,R)>0$ so that $ (B+g_l^k)\cap (B+g_m^k)=\emptyset$ or every $k\ge k_0$  and $\dist(B+ g^k_l,B+g^k_m)\to \infty$ for every $l\neq m,$ with $l,m \in \{1,...,M\}$. Thus, by monotonicity, additivity of the perimeter measure on disjoint sets and $G$-invariance, we have
    \begin{align*}
        \Per(E^k) &\ge \Per(E^k,\cup_{l=1}^M(B+g^k_l)) \\
        &\ge \sum_{l=1}^M\Per(E^k-g_k^l,B) - \sum_{l=1}^M |E^k \cap ( B+g^k_l)  |\min_{\substack{m\neq l,\\ m \in (1,...,M)}} \phi({\rm dist}(B+g^k_m,B+g^k_l )),
    \end{align*}
    for all $k\ge k_0$, where $\phi(\cdot)$ can be taken to zero in the anisotropic case. Observe that in the non-local case, we have $\phi( {\rm dist}(B+g^k_m,B+g^k_l ) )\to 0$ as $k\uparrow \infty$ for all $m\neq l\le M$, by definition of $\phi(\cdot)$ (recall Section \ref{sec:perimeters}).  We can thus take the limit as $k\uparrow\infty$ and estimate for every $M\in\N$
    \begin{align*}
        \liminf_{k\uparrow\infty}  \Per(E^k)  \ge \sum_{l=1}^M \liminf_{k\uparrow\infty}\Per(E^k-g^k_l,B) \ge \sum_{l=1}^M\Per(E_l,B) \ge \Per(\cup_{l=1}^ME_l,B),
    \end{align*}
    by Fatou and lower semicontinuity with respect $L^1_{loc}$-convergence. Finally, using that  $\cup_{l=1}^M E_{l} \to \cup_l E_{l}$  monotonically as $M\uparrow\infty$, we can first send $M$ to infinity using again lower semicontinuity and then conclude by monotonicity on $B=B_R(0)$ sending afterwards $R$ to infinity.
\end{proof}
We discuss an example where a minimal partition can be exactly found.
\begin{remark}
    \rm
    Let $\varphi \colon \R^d \to [0,\infty)$ be a norm and let $\Per_\varphi$ be the associated anisotropic perimeter. Denote by $W\subseteq \R^d$ the Wulff shape with $|W|=1$ minimizing the perimeter among sets with unit volume. 

    Let $L>0,\mu = 0$.  Fix $N \in \N$ and a volume vector $ \vec v \coloneqq (v_1,...,v_N)$ so that $v_i>0, \sum_i v_i = L^d$. Suppose that there exist a lattice $G$, a fundamental domain $D$ for the action of $G$, elements  $g_i \in G$ and scalings $\lambda_i>0$ so that for all $i,j=1,...,N$ with $i\neq j$
    \[
    |\lambda_i W| = v_i,\quad |D \setminus  \cup_{i=1}^N (\lambda_i W +g_i)|=0,\quad |(\lambda_i W + g_i)\cap (\lambda_j W + g_j)|=0. 
    \]
    Then, $G, D$ and $E_i \coloneqq \lambda_i W +g_i$ is a minimizer in Theorem \ref{thm:step 2} and Theorem \ref{thm:step 3}.  This clearly follows by observing that, for any other competitor $\tilde G,\tilde D,(\tilde E_i)$ with $|\tilde E_i|=v_i$, we have
    \[
    \sum_i \Per_\varphi(\tilde E_i) \ge \sum_i \Per_\varphi( \lambda_i W + g_i).
    \]
    
Finally, a concrete example is the case of $\varphi (\nu ) = \sum_{j=1}^d|\nu_j|$ for every $\nu \in \R^d$ and $\vec v$ so that a partition $(E_i)$ of the cube $L\cdot Q$ exists with each $E_i$ being a translation and dilation of the unit cube $Q$. As $W=Q$ is the Wulff shape associated with this norm with unit volume, the above discussion applies.  Here, the relevant fact is that it is possible to tile the space with squares, and each square is of minimal perimeter for its volume.
\end{remark}

\subsection{A penalized problem}
In this part, we study a penalized version of the minimization problems already faced by relaxing the volume constraint. 
\begin{theorem}\label{thm:penalized}
    Let $L>0, \mu\ge 0,\lambda >0$.  Fix $N \in \N$ and a volume vector $ \vec v \coloneqq (v_1,...,v_N)$ so that $v_i>0, \sum_i v_i = L^d$.  Consider the minimization problem
    \[
        P_*\coloneqq  \inf_G\inf_D \inf \left\{ \mu \Per(D)+ \sum_i\frac 12 \Per(E_i) + \lambda \big| |E_i|-v_i\big|  \colon \begin{array}{l}
        E_i\subseteq D\emph{ Borel} \\ |E_i\cap E_j|=0 \\
        \big|D\setminus \cup_i E_i\big|=0
    \end{array}  \right\},
    \]
     where the last two infima are respectively taken among all lattices $G$ with $d(G^k)=L^d$ and all fundamental domains $D$ for the action of $G$ and $\Per$ is either the local or nonlocal perimeter.

    Then, there exist a lattice $G$, a fundamental domain $D$ for the action of $G$ and a partition of $(E_i)$ admissible so that 
    \[
       P_* =  \mu \Per(D)+ \sum_i \frac 12\Per(E_i) + \lambda \big| |E_i|-v_i\big| .
    \]
\end{theorem}

\begin{proof}
   For every $k \in \N$ we can consider lattices $(G^k)$,  fundamental domains $(D^k)$ for the action of $G^k$, and a partition $(E^k_i)$ so that
    \[
   P_* \ge \mu \Per(D^k) +\sum_i \frac{1}{2}\Per(E^k_i)+\lambda \big||E_i^k|-v_i\big|- \frac 1k,\qquad \forall k \in\N.
   \]
   In particular, we have the following uniform bounds
  \begin{align}
     &L^d = d(G^k)=|D^k|, & &\sup_{k\in\N}|E_i^k|<\infty, \quad\forall i,\label{eq:bounds Vol penalized} \\
     &\sup_{k\in\N}\Per(D^k) <\infty,&& \sup_{k\in\N}\Per( E^k_i) <\infty, \quad \forall i.\label{eq:bounds Per penalized}
 \end{align} 
    We are in the position to invoke the concentration compactness result in \cite[Lemma 3.1]{CesaroniFragalaNovaga23} (again, \cite[Lemma 3.2,Theorem 3.3]{CesaroniNovaga23_1} for the local anisotropic perimeter) to deduce, up to a non relabeled subsequence, the existence of a lattice $G$ with $d(G)=L^d$, of $(g_l^k)_l \subseteq G^k$ and of Borel sets $D_l\subseteq \R^d$ so that
    \begin{align*}
        \sum_l|D_l|= L^d,\qquad D^k-g_l^k \to D_l\quad\text{in }L^1_{loc}\text{ as $k\uparrow\infty$ for all $l$},
    \end{align*}
    and also that $|D_l\cap D_m|=0$ and $D\coloneqq \cup_lD_l$ is a fundamental domain for the limit lattice $G$. By a diagonalization argument and precompactness of sets of finite perimeters guaranteed  by \eqref{eq:bounds Vol penalized},\eqref{eq:bounds Per penalized}, we can pass to a common further subsequence in $k$ so that
    \[
    (E^k_i - g_l^k ) \to E_{i,l},\qquad\text{ in $L^1_{loc}$ as }k\uparrow\infty\text{ and for all }i,l,
    \]
    for some Borel sets $(E_{i,l})\subseteq \R^d$. Set now
    \[
     E_i \coloneqq \bigcup_l E_{i,l},\quad \forall i.
    \]
    Arguing now as previously done in the proof of Theorem \ref{thm:step 3}, we get that $(E_i)$ is admissible in the definition of $P_*$ and also that $|D\setminus \cup_i E_i|=0$ as well as $|E_i\cap E_j|=0$ for $i\neq j$. In particular, we also have $\sum_i |E_i|=L^d$.
    
    To conclude the proof we thus only need to prove that
    \begin{equation}
    P_* =  \mu \Per(D) +\sum_i \frac{1}{2}\Per(E_i)+\lambda \big||E_i|-v_i\big|
    \label{eq:claim conclusion}
    \end{equation}
    This however follows by the same arguments already employed in the proof of Theorem \ref{thm:step 3}, taking into account the following claim due to the relaxation of the volume constraint
    \begin{equation}
    \big||E_i|-v_i\big|  =\lim_{k\uparrow\infty}\big||E^k_i|-v_i\big|,\qquad \forall i.
    \label{eq:volume continuity}
    \end{equation}
    To see this, fix any $i = 1,...,N$ and notice that
    \begin{align*}
         |E_i| &\le \sum_l |E_{i,l}|  = \sum_l |E_{i,l}\cap D|  \le \liminf_{k\uparrow\infty}\sum_l|(E_i^k -g^k_l)\cap D)| \\
         &\le \liminf_{k\uparrow\infty}\sum_l|E_i^k \cap (D +g_l^k)|\le \liminf_{k\uparrow\infty}\sum_{g \in G}|E_i^k \cap (D +g)| = \liminf_{k\uparrow\infty}|E_i^k|,       
    \end{align*}
    since $D$ is fundamental. However, using the identity $\sum_i |E_i|=L^d = \sum_i |E^k_i|$ for every $k\in \N$, we get from the above that
    \[
       L^d = \sum_i |E_i| \le \sum_i\liminf_{k\uparrow\infty}|E_i^k|\le \liminf_{k\uparrow\infty}\sum_i |E_i^k| = L^d.
    \]
    This gives that all the liminf in the above are limits, and consequently that
    \[
        |E_i^k|\to |E_i|,\qquad\text{as $n\uparrow\infty$ for all i},
    \]
    must hold that implies \eqref{eq:volume continuity}.  Finally, the claim \eqref{eq:claim conclusion} and thus the conclusion of the proof follows now by the lower semicontinuity result of Lemma \ref{lem:semicontinuity} combined with \eqref{eq:volume continuity}.
\end{proof}
\section{Local minimality and regularity}
We here study local minimality and possibly the regularity properties for the periodic partitions of $\R^d$ given by the previous section.
\subsection{Local minimality for volume constrained problem}\label{sec:minimality}
\begin{definition}
    Let $\Per$ be either the local or the non-local perimeter and $\Lambda\ge 0$. We say that a partition $(E_i)$ of $\R^d$ is a \emph{volume constrained $\Lambda$-minimizer} of $\Per$ in an open set $\Omega \subseteq \R^d$, provided $\sum_i\Per(E_i,\Omega)<\infty$ and for any other partition $(F_i)$ of $\R^d$ with the property that $ E_i \triangle F_i \Subset \Omega$ and $|E_i \cap \Omega|=|F_i\cap \Omega|$ for any $i$, it holds
    \[
    \sum_i\Per(E_i,\Omega) \le\sum_i \Per(F_i,\Omega) + \Lambda |E_i\triangle F_i|.
    \]
    Moreover, for $r>0$ we say that $(E_i)$ is \emph{volume constrained $(\Lambda,r)$-minimal}, provided it is a volume constrained $\Lambda$-minimizer on every ball $B_r(x)$ for $x \in \R^d$.
\end{definition}
We have the following minimality property: 
\begin{lemma}\label{lem:minimality constrained}
    Let $L>0,N\in\N$  and a volume vector $ \vec v \coloneqq (v_1,...,v_N)$ so that $v_i>0, \sum_i v_i = L^d$. Consider  $G,D,(E_i)$ minimizers in Theorem \ref{thm:step 3}, for $\mu =0$. Then, the partition $(E_i + g)_{i,g \in G}$ of $\R^d$ is a volume constrained $(\Lambda,r)$-minimal of the local anisotropic perimeter for every $r< \rho_G/2 $ (recall $\rho_G$ is the packing radius of $G$) with $\Lambda=0$ in the case of the local perimeter and $\Lambda := \int_{\R^d\setminus B_{\rho_G-2r}(0)}K(x)\, \d x$ in the case of the non-local perimeter.
\end{lemma}
\begin{proof}
     Recall that $D =\cup_i E_i$ is a fundamental domain for the action of $G$. Let $(g_l)$ for $l \in\N$ be an enumeration of $G$.  Let us denote for brevity $E_{i,l} \coloneqq E_i + g_l$ for every $l \in \N,i$ and $B\coloneqq B_r(x)$, for $x \in \R^d,r<\rho_G/2$.

    Let us consider $F_{i,l}$ another partition of $\R^d$ with the property that $E_{i,l}\triangle F_{i,l}\Subset B$ and $|E_{i,l}\cap B|=|F_{i,l}\cap B|$ for all $i,l$. Set
    \[
    F_i \coloneqq  \left( E_i \setminus \bigcup_{ l} (B-g_l)\right)\cup \left(\bigcup_{l} (F_{i,l} -g_l)\right).
    \]
    We claim that $\tilde D \coloneqq \cup_i F_i$ is a fundamental domain for the action of $G$. This simply follows by the fact that $D$ is fundamental and that $(D +g_l) \triangle (\cup_iF_{i,l}) \Subset B$ by assumptions (recall that $\#\{i\} = L^d$). Next, we claim that $|F_i|=1$ for every $i$. This easily follows combining 
    \[
    1=|E_i|,\qquad 0=|E_i\cap E_j|=|E_{i,l}\cap E_{j,l}|
    \]
    with
    \[
    |E_i \cap (B-g_l)|=|E_{i,l}\cap B|=|F_{i,l}\cap B| = |(F_{i,l} -g_l)\cap (B-g_l)|,
    \]
    thanks to the assumptions on $F_{i,l}$.

Let now $\phi:(0,+\infty)\to (0,+\infty)$ be defined as $\phi\equiv 0$ in the case of the local perimeter, and
\[
\phi(t) := \int_{\R^d\setminus B_{t}(0)}K(x)\, \d x
\]
in the case of the non-local perimeter.
    
    The fundamental domain $\tilde D$ and its partition $F_i$ are admissible for the definition of $f_0(\cdot)= f(\cdot)$. Thus, by minimality, we can estimate
    \begin{align*}
         0 & \le f(\tilde D)-f(D) = \sum_i\Per(F_i) -\sum_i\Per(E_i) \\
         &= \sum_i\Per\big(F_i,\cup_l (B-g_l)\big) -\sum_i\Per\big(E_i, \cup_l (B-g_l)\big) +  \sum_{i,l} \phi (\rho_G-2r)|(E_i+g_l)\triangle F_{i,l} | \\ 
         &\le \sum_{i,l}\Per(F_{i,l},B)-  \sum_{i,l}\Per((E_i + g_l), B) + \sum_{i,l} \phi (\rho_G-2r)|(E_i+g_l)\triangle F_{i,l} | ,
    \end{align*}
    having used the subadditivity of the perimeter on disjoint sets and its translation invariance, as well as $r<\rho_G/2$ and that $\phi$ is monotone non-decreasing (possibly zero, for the local perimeter). This concludes the proof.
\end{proof}
As anticipated in the Introduction, we cannot prove regularity for volume constrained $(\Lambda,r)$-minimal partitions in the above sense,
since we cannot prove that the partition is locally finite.
However, in the planar case, we shall see that the situation simplifies and regularity can be investigated. 
More generally, we can study regularity for the penalized problem removing the volume constraint as we are going to see next.
\subsection{Local minimality and regularity for the penalized problem}
We here study local minimality and regularity properties for the periodic partitions of $\R^d$ given by the penalized problem in Theorem \ref{thm:penalized}. We start with a definition of local minimality.
\begin{definition}
    Let $\Per$ be either the local or the non-local perimeter and $\Lambda\ge 0$. We say that a partition $(E_i)$ of $\R^d$ is a \emph{$\Lambda$-minimizer} of $\Per$ in an open set $\Omega \subseteq \R^d$, provided $\sum_i\Per(E_i,\Omega)<\infty$ and for any other partition $(F_i)$ of $\R^d$ with the property that $ E_i \triangle F_i \Subset \Omega$ for any $i$, it holds
    \[
    \sum_i\Per(E_i,\Omega) \le\sum_i \Per(F_i,\Omega) + \Lambda |E_i\triangle F_i|.
    \]
    Moreover, for $r>0$ we say that $(E_i)$ is \emph{$(\Lambda,r)$-minimal}, provided it is a volume constrained $\Lambda$-minimizer on every ball $B_r(x)$ for $x \in \R^d$.
\end{definition}
We have the following minimality property. 
\begin{lemma}\label{lem:minimality penalized}
   Let $L>0,N\in\N$  and a volume vector $ \vec v \coloneqq (v_1,...,v_N)$ so that $v_i>0, \sum_i v_i = L^d$. Consider $G,D,(E_i)$ minimizers in Theorem \ref{thm:penalized} for $\lambda>0,\mu =0$. Then, the partition $(E_i + g)_{i,g \in G}$ of $\R^d$ is $(\Lambda,r)$-minimal for every $r< \rho_G/2 $ (recall $\rho_G$ is the packing radius of $G$) where $\Lambda:=\lambda $ in the case of the local perimeter and 
   $\Lambda := \lambda + \int_{\R^d\setminus B_{\rho_G-2r}(0)}K(x)\, \d x$ in the case of the non-local perimeter.
\end{lemma}
\begin{proof}
    Recall that $D =\cup_i E_i$ is a fundamental domain for the action of $G$. Let $(g_l)$ for $l \in\N$ be an enumeration of $G$.  Let us denote for brevity $E_{i,l} \coloneqq E_i + g_l$ for every $l \in \N,i$ and $B\coloneqq B_r(x)$, for $x \in \R^d,r<\rho_G/2$.

    Let us consider $F_{i,l}$ another partition of $\R^d$ with the property that $E_{i,l}\triangle F_{i,l}\Subset B$ for all $i,l$. Set
    \[
    F_i \coloneqq  \left( E_i \setminus \bigcup_{ l} (B-g_l)\right)\cup \left(\bigcup_{l} (F_{i,l} -g_l)\right).
    \]
    We claim that $\tilde D \coloneqq \cup_i F_i$ is a fundamental domain for the action of $G$. This simply follows by the fact that $D$ is fundamental and that $(D +g_l) \triangle (\cup_iF_{i,l}) \Subset B$ by assumptions.

    The fundamental domain $\tilde D$ and the partition $F_i$ are admissible and hence by minimality, we can estimate
    \begin{align*}
         0 & \le \frac 12 \sum_i\Per(F_i) + \lambda \big||E_i|-v_i\big| -\frac 12\sum_i\Per(E_i) -\lambda \big||F_i|-v_i\big| \\
         &\le\sum_i\Per\big(F_i,\cup_l (B-g_l)\big) -\sum_i\Per\big(E_i, \cup_l (B-g_l)\big) +\lambda \sum_i \big| E_i\triangle F_i\big|\\
         &\qquad +\sum_{i,l} \phi (\rho_G-2r)\big|((E_i+g_l)\triangle F_{i,l}) \cap B \big| \\
         &\le  \sum_{i,l}\Per(F_{i,l},B)-  \sum_{i,l}\Per((E_i + g_l), B) +\left(\lambda +\phi (\rho_G-2r)\right) \sum_{i,l} \big|( (E_i +g_l)\triangle F_{i,l})\cap B \big| \\
    \end{align*}
    having used the subadditivity of the perimeter on disjoint sets, its translation invariance, the estimate $\big||E|-c\big| -\big||F|-c\big| \le \big| E\triangle F\big|$ by triangular inequality for $c \in \R$ and every $E,F\subseteq \R^d$ and  $r<\rho_G/2$ and that $\phi$ is monotone non decreasing (possibly zero, for the local perimeter). This concludes the proof.    
\end{proof}
Thanks to the above, we can study the $C^{1,\alpha}$-regularity of boundaries in the local anisotropic case.
\begin{proposition}\label{prop:regularty penalized 1}
     Let $\Per_\varphi$ be the anisotropic perimeter, $L>0,N\in\N$  and let $ \vec v \coloneqq (v_1,...,v_N)$ be a volume vector so that $v_i>0, \sum_i v_i = L^d$. Consider $G,D,(E_i)$ minimizers in Theorem \ref{thm:penalized} for $\lambda>0,\mu =0$. Then, for all $i=1,...,N$, we have that $\partial E_i$ is $C^{1,\alpha}$ regular up to a singular set $\Sigma$ that is $\HH^{d-1}$-negligible.
\end{proposition}
\begin{proof}
    First, notice that $\HH^{d-1}(\partial E\setminus \partial^* E) =0$ for every set of finite perimeter $E\subseteq \R^d$, where $\partial^* E$ is the \emph{reduced boundary}. See \cite{AmbrosioFuscoPallarabook} for this claim and the related notions. Therefore, we can equivalently prove regularity of $\partial^*E_i$. 
    
    Let us now  consider the $G$-periodic partition $ (E_i + g_l)_{i,l}$ of $\R^d$, where $g_l$ is an enumaration of $G$. The conclusion will follow by proving that there exists some structural radius $r>0$ independent on $g \in G$ so that for any $i=1,...,N$ and for $\HH^{d-1}$ every $x \in \partial^* E_i$, we have that $B_r(x) \cap \partial^*E_i $ is $C^{1,\alpha}$ outside of a $\HH^{d-1}$-negligible set.  Indeed, consider (without loss of generality and reordering indexes) $x \in \partial^* E_1 \cap \partial^*(E_2 + g)$ for a suitable $g \in G$. Then, by Lemma \ref{lem:minimality penalized}, we know that the partition $ (E_i + g_l)_{i,l}$ is $(0,s)$-minimal for all $s<\rho_G/2$. Thus, the Elimination Lemma \cite[Lemma 4.4]{CesaroniNovaga22} applies at balls centred at $x$, giving the existence of $\sigma_0,r_0>0$ (depending only on $N,L,n$) so that  for every $s<\min\{ r_0,\rho_G/2\}$ it holds
    \[
    |\cup_{i>2,l} (E_i +g_l) \cap B_s(x)|\le \sigma_0r^d \quad \implies \quad |\cup_{i>2,l} (E_i +g_l) \cap B_{s/2}(x)| =0.
    \]
    Therefore, since $x \in \partial^* E_1 \cap \partial^*(E_2 + g)$, then by definition of reduced boundary and anisotropic perimeter and choosing $r=s/2$, we see that Lemma \ref{lem:minimality penalized} yields in this case
    \[
    \Per_\varphi (E_1,B_r(x))\le \Per_\varphi (F,B_r(x)),
    \]
    for all $F\subseteq \R^d$ so that $F\triangle E_1 \Subset B_r(x)$. Thus, the classical regularity theory \cite{Bombieri82} applies and by arbitrariness of $x$ the proof is then concluded.
\end{proof}

We conclude this part by stating, without proof,  local finiteness and higher order regularity results in the case of classical and fractional perimeter

The proof of the following can be obtained by the same arguments already presented in \cite[Theorem 4.9]{CesaroniNovaga22} (extending \cite{Choe89}) and it is based on the monotonicity formula by \cite{MassariTamanini91} and the local finiteness result for conical partitions (cf. \cite[Proposition 4.5]{CesaroniNovaga22}). This reasoning makes it possible to deduce local finiteness of the minimal partition and therefore reduce it to classical regularity theory \cite{Maggi12_Book}.
\begin{proposition}\label{prop:regularty penalized 2}
     Let $\Per$ be the classical perimeter, $L,N\in\N$  and let $ \vec v \coloneqq (v_1,...,v_N)$ be a volume vector so that $v_i>0, \sum_i v_i = L^d$. Consider $G,D,(E_i)$ minimizers of Theorem \ref{thm:penalized} for $\lambda>0,\mu =0$. Then, the partition $(E_i+g)_{i,g\in G}$ is locally finite and hence $D$ is bounded. Moreover, for all $i=1,...,N$, we have that $\partial E_i$ a $C^\infty$-hypersurface in $\R^d$ up to a $\HH^{d-1}$-negligible closed singular set $\Sigma \subseteq \R^d$. Finally, if $d=2$, then $\Sigma$ is discrete. 
\end{proposition}

For the fractional perimeter, the proof of the following can be obtained by the same arguments already presented in \cite[Theorem 4.10]{CesaroniNovaga22} based on the monotonicity formula \cite{ColomboMaggi2017} and again the local finiteness result for conical partitions (cf. \cite[Proposition 4.5]{CesaroniNovaga22}).  This reasoning makes it possible to deduce local finiteness of the minimal partition and therefore reduce to apply the regularity theory in \cite{Maggi12_Book}. Higher order regularity then follows by a bootstrap argument as performed in \cite[Theorem 4.10]{CesaroniNovaga22}
 \begin{proposition}\label{prop:regularity penalized 3}
    Let $\Per$ be the fractional perimeter with interaction kernel
    \[
        K(x) \coloneqq \frac{1}{|x|^{d-s}},\qquad x \in \R^d,s\in(0,1).
    \]
    Let $L>0,N\in\N$  and a volume vector $ \vec v \coloneqq (v_1,...,v_N)$ so that $v_i>0, \sum_i v_i = L^d$. Consider $G,D,(E_i)$ minimizers in Theorem \ref{thm:penalized} for $\lambda>0,\mu =0$.  Then, the partition $(E_i+g)_{i,g\in G}$ is locally finite and hence $D$ is bounded. Moreover, for all $i=1,...,N$, we have that $\partial E_i$ a $C^\infty$-hypersurface in $\R^d$ up to a $\HH^{d-1}$-negligible closed singular set $\Sigma \subseteq \R^d$. Finally, if $d=2$, then $\Sigma$ is discrete.
 \end{proposition}

\section{The planar case}
We now specialize our investigation on the planar case. We first address some basic regularity results and then prove a stability result of the Honeycomb tessellation.
\subsection{Local finiteness and regularity}
We start by proving the existence of a locally finite minimal $G$-periodic planar partition in Theorem \ref{thm:step 3}.
\begin{proposition}\label{prop:planar local finiteness}
    Let $\Per_\varphi$ be the local anisotropic perimeter in $\R^2$. Let $L>0,N\in\N$  and let $\vec v \coloneqq (v_1,...,v_N)$ 
    be a volume vector so that $v_i>0, \sum_i v_i = L^2$. Then, there exists a lattice $G$, a fundamental domain $\tilde D$ for the action of $G$ and a partition $(\tilde E_i)$ minimizing Theorem \ref{thm:step 3} for $\mu\ge 0$ so that 
    \[
     \tilde D \cap B_{r_G}(0) \neq \varnothing,\qquad \text{and}\qquad  {\rm diam}(\tilde D)\le 2r_G +C,
    \]
    for some constant $C>0$ depending on $\mu,L,N$. Finally, the $G$-periodic partition $(\tilde E_i+g)_{i,g\in G}$ of $\R^2$ is locally finite.
\end{proposition}
\begin{proof}
    Let us consider a minimizing lattice $G$, a fundamental domain $ D$ for the action of $G$ and a partition $(E_i)$ given by Theorem \ref{thm:step 3}. Since $D$ is so that $\Per_\varphi(D)<\infty$, we can  consider (\cite{AmbrosioCasellesMasnouMorel01}) possibly countable indecomposable components $D_l\subset D$, i.e. so that
    \[
    D =  \cup_l D_l, \qquad |D_l\cap D_,|=0, \quad \forall l\neq m,\qquad \Per_\varphi(D)=\sum_l\Per_\varphi (D_l).
    \]
    Up to redefining each $D_l$ possibly discarding negligible sets, we might suppose that
    \[
        D_l\cap D_m =\emptyset \qquad D_l\cap ( D_m + g )=\emptyset , \qquad \forall l\neq m,\, g \in G.
    \]    
    Thus, by the diameter-perimeter estimate for indecomposable planar sets of finite perimeter \cite[Lemma 2.13]{DayrensMasnouNovagaPozzetta22} we know for some $c>0$ that
    \[
    {\rm diam}(D_l)\le c \Per_\varphi(D_l),\qquad \forall l.
    \]
    The above properties have been established for the classical perimeter in which case $c = 1/2$. Clearly, these carry over for the local anisotropic perimeter up to a constant $c>0$.

    Let us now consider $g_l \in G$ so that $(D_l-g_l)\cap B_{r_G}(0) \neq \varnothing$ (here $r_G>0$ is the covering radius of the lattice $G$) and define the sets
    \[
       \tilde E_i \coloneqq \cup_l \big((E_i \cap D_l)-g_l\big),\qquad \tilde D_l \coloneqq D_l -g_l, \quad \forall i,l.
    \]
    By construction, $\tilde D \coloneqq \cup_l \tilde D_l$ is a fundamental domain for the action of $G$ and $(\tilde E_i)$ is a partition up to negligible sets of $\tilde D$. We now claim that $G,\tilde D, (\tilde E_i)$ are so  that
    \[
        \mu\Per_\varphi(\tilde D) + \frac 12\sum_i\Per_\varphi(\tilde E_i) =   \mu\Per_\varphi( D) + \frac 12\sum_i\Per_\varphi( E_i),
    \]
    This would give that $G,\tilde D, (\tilde E_i)$ are also minimizers in Theorem \ref{thm:step 3}. Notice that the $\ge$ inequality in the above is trivial, as we have just shown that these are competitors for the minimization problem in Theorem \ref{thm:step 3}. For the converse, we estimate
    \[
        \Per_\varphi(D) = \sum_l \Per_\varphi(D_l) = \sum_l \Per_\varphi(D_l-g_l) \ge \Per_\varphi(\cup_l(D_l-g_l)) = \Per_\varphi(\tilde D).
    \]
    and for all $i$ that
    \begin{align*}
        \Per_\varphi(E_i)&= \Per_\varphi(E_i\cap D) =\sum_l\Per_\varphi(E_i \cap D_l) =\sum_l\Per_\varphi((E_i \cap D_l) -g_l) \ge \Per_\varphi(\tilde E_i),
    \end{align*}
    having used, in both, that $(D_l)$ is an indecomposable decompisition of $D$ and the $\sigma$-subadditivity and $G$-invariance properties of the anisotropic perimeter. Notice that, by construction, we have for all $i$
    \[
    {\rm diam}(\tilde E_i)\le {\rm diam}(\tilde D) \le 2r_G + \sum_l {\rm diam}(D_l) \le  2r_G + c \Per_\varphi(D).
    \]
    The first claim then follows by choosing $C\coloneqq C(\mu,L,N)>0$ so that $c\Per_\varphi(D)\le C$, which can be obtained considering the lattice competitor $L\cdot \Z^2$ and arguing as for \eqref{eq:C mu N}.
    
    To conclude, we need to show that the partition of $\R^2$ given by $(\tilde E_i + g)_{i,g\in G}$ is locally finite. To this aim, we will prove that for each $K\subseteq \R^2$ compact, it holds
    \[
     \# J<\infty, \quad \text{for}\quad J \coloneqq \{ g \in G \colon (\tilde D+g)\cap K \neq \emptyset\}. 
    \]
    Indeed, since $\tilde D$ is pre-compact, as it is bounded, then ${\rm cl}(\tilde D)\cup K \cap \big( {\rm cl}(\tilde D)\cup K + g\big) \neq \emptyset$ occurs only for finite many $g \in G$, being $G$ properly discontinuous. It is evident that any $g \in J$ is so that the latter property is true, concluding the proof.
\end{proof}
Thanks to the local finiteness of the minimal partition in the planar case, we are able to reduce the regularity problem to classical results in the theory of clusters \cite{Almgren76} (see also \cite{Morgan94,MorganBOOK,Maggi12_Book}). We shall need here the version stated in \cite{NovagaPaoliniTortorelli23} for locally minimizing clusters relative to open sets.
\begin{theorem}\label{thm:regularity constrained planar}
   Let $\Per$ be the classical perimeter in $\R^2$. Let $L>0,N\in\N$  and let $ \vec v \coloneqq (v_1,...,v_N)$ be a volume vector so that $v_i>0, \sum_i v_i = L^2$. Consider $G,D,(E_i)$ minimizers in Theorem \ref{thm:step 3} with $\mu= 0$. Then, for every $x \in\R^2$ and $r<\rho_G/2$, it holds
    \begin{itemize}
        \item[{\rm i)}] $\cup_{i,g} \partial^* (E_i +g)  \cap B_r(x)$ is made of a finite number of straight lines or circular arcs meeting at triples with angles $120^\circ$;
        \item[{\rm ii)}] the three signed curvature of arcs in $\cup_{i,g} \partial^* (E_i +g)  \cap B_r(x)$ meeting in a vertex have zero sum;
        \item[{\rm iii)}] is it possible to define a finite number of pressures $\rho_{i,g}$ so that the curvature separating $(E_i+g)$ and $(E_j +g')$ is precisely $\rho_{i,g}-\rho_{j,g'}$ if $i\neq j$ or zero if $i=j$ (with the convention that the pressure $\rho_{i,g}$ is positive when the arc has a concavity towards $E_i+g$).
    \end{itemize}
\end{theorem}
\begin{proof}
    Let us denote by $G,\tilde D,(\tilde E_i)$ the minimizers constructed in Proposition \ref{prop:planar local finiteness} from $G,D,(E_i)$. We will prove first that the partition $(\tilde E_i +g)_{i,g}$ of $\R^2$ satisfies all the listed regularity properties, and then argue that also the original partition $(E_i+g)_{i,g}$ shares the same properties.

    Let $(g_m)$ be an enumeration of $G$. Recall from Lemma \ref{lem:minimality constrained} that the partition $(\tilde E_i +g_m)_{i,m}$ of $\R^2$ is $(0,r)$-minimal for every $r<\rho_G/2$, that is for every $x \in \R^2$ and $r < \rho_G/2$, we have
    \[
    \sum_{i,m} \Per(\tilde E_i+g_m,B_r(x))\le \sum_i\Per(F_{i,m},B_r(x)),
    \]
    for all partitions $(F_{i,m})$ of $\R^2$ with $(\tilde E_i +g_m)\triangle  F_{i,m} \Subset B_r(x)$ and $|(\tilde E_i+g_m)\cap \tilde B_r(x)|=|F_{i,m}\cap B_r(x)| $. 

    However, by Proposition \ref{prop:planar local finiteness}, we also know that 
    \[
    \tilde D\text{ is bounded,}\quad \text{and}\quad \#  \{ m \colon (\tilde E_i + g_m) \cap B_r(x)\}<\infty,\qquad \forall i.
    \]
    Thus, there exist $J_i\subseteq \N$ finite (possilbly empty for some $i=1,...,N$, depending on $r,x$) so that 
    \[
    (\tilde E_i + g_m)_{i \in \{1,...,N\}, l \in J_i}\qquad \text{is a partition up to negligible sets of }B_r(x),
    \]
    and thus it is locally minimal on $B_r(x)$.

    We are in a position to invoke the classical regularity theory for minimal clusters, referring in the localized version to open sets in the planar case to \cite[Theorem 4.1]{NovagaPaoliniTortorelli23}. We thus deduce that all the listed regularity properties for the partition $\{\tilde E_i + g_m\}_{i,m}$  on the open set $B_r(x)$. 

    To conclude, we notice that, since $\tilde D$ as well $\tilde E_i$ are regular, the indecomposable components $D_l$ of $D$ must be of finite number (recall from the proof of Proposition \ref{prop:planar local finiteness} that $\tilde D,\tilde E_i$ are achieved by possibly countable $G$-translations $D_l$ along suitable elements of $G$). In particular, $D,(E_i)$ must be also bounded and accordingly the partition $(E_i+g_m)_{i,m}$ of $\R^2$ is locally finite. Therefore, all the previous arguments apply as well to the original partition yielding the conclusion of the proof.
\end{proof}
\subsection{Stability of Honeycomb for almost equal areas}
In this part, we prove a stability result of the standard Honeycomb tessellation of the plane. In our formalism, we consider $N$ equal regular hexagons of unit volume attached as in Figure \ref{fig:Honeycomb} and their union will consist in a fundamental domain $D_{\sf HC}$ for the $N$-Honeycomb lattice $G_{\sf HC}$ with volume $d(G_{\sf HC})=N=|D_{\sf HC}|$. 
\begin{figure}[!ht]
    \centering
    \includegraphics[scale =.7]{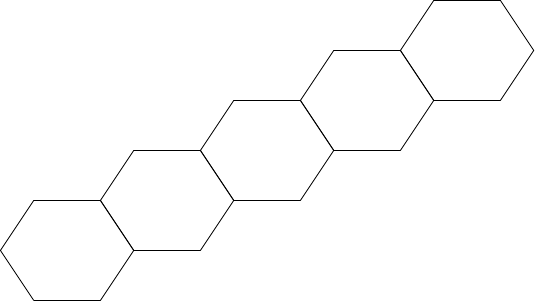}
    \caption{Fundamental domain $D_{\sf HC}$ for the action of the lattice $G_{\sf HC}$ associated to $N$ regular hexagons with unit volumes}
    \label{fig:Honeycomb}
\end{figure}
In a nutshell, we prove that when our minimal partitions have $N$ chambers with \emph{almost} all the same amount of area, then the minimal energy in Theorem \ref{thm:step 3} is the same of that of $N$ regular hexagons and a minimal partition is given by a slight modification of this reference configuration. 

Consider $N\in\N$ and a volume vector $\vec v\coloneqq (v_1,...,v_N)$ so that 
\[
 \sum_{i=1}^Nv_i = N,\qquad v_i \in \left( \frac 12, \frac 32 \right),\quad\text{for all }i=1,..,N,
\]
in particular, we will consider $L\coloneqq \sqrt N$. 

A competitor in Theorem \ref{thm:step 3} is given by the following choice:
\begin{equation}
 G_{\sf HC}, \qquad D(\vec v)\coloneqq \cup_i H(v_i) ,\qquad E_i\coloneqq H(v_i).
\label{eq:competitor H}
\end{equation}
where $H(v_i)\subseteq \R^2$ are defined iteratively as follows. First, $H(v_1)$ is the hexagon with $|H(v_1)|=v_1$ obtained from the standard one with unit area by stretching only the horizontal edges by a factor $x_1$ so to match the area, see Figure \ref{fig:pic3}.
\begin{figure}[!ht]
    \centering
    \includegraphics[scale =1.1]{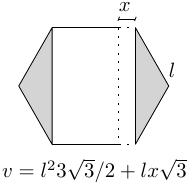}
    \caption{The area $v$ of an hexagon obtained from the regular hexagon with edge $l$ stretching the horizontal edge by a factor $|x|<l$.}
    \label{fig:pic3}
\end{figure}
Then, in general, $H(v_i)$ is the hexagon with $|H(v_i)|=v_i$ obtained stretching accordingly the horizontal edges by a factor $x_i$ and with a common oblique edge with $H(v_{i-1})$, for all $i=2,...,N$. See Figure \ref{fig: D with H v} for the construction and notice that it is well defined as we require $v_i \in \big(1/2,3/2\big)$, hence $|x_i|>l$ never happens as a choice in Figure \ref{fig:pic3} since the overall amount of grey area is equal to $1/2$. When $v=1$, we simply write $H(1) \coloneqq H$, that is, the regular hexagon with unit area.
Notice that the set $D(\vec v)\coloneqq \cup_i H(v_i)$ is a fundamental domain for the reference lattice $G_{\sf HC}$.

\begin{figure}[!ht]
    \centering
    \includegraphics[scale =0.8]{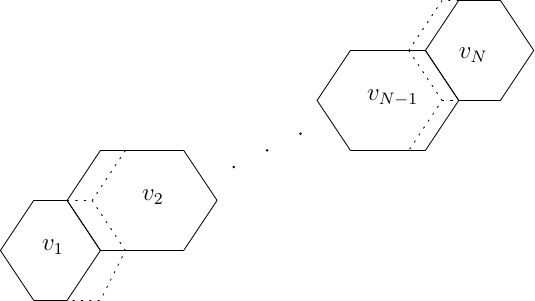}
    \caption{Partition of $D(\vec v)$ with unequal prescribed areas.}
    \label{fig: D with H v}
\end{figure}

It is immediate that, with this choice, \eqref{eq:competitor H} are well defined and competitors. Furthermore, since we require $\sum_i v_i = N$, then $\sum_i x_i =0$ and therefore it holds
\[
\frac 12 \sum_{i=1}^N \Per(H(v_i)) =  \frac 12 \sum_{i=1}^N \big( \Per(H) +2x_i\big) = \frac N2 \Per(H),
\]
where $\Per$ is the classical perimeter in $\R^2$. We are now ready to prove our main stability result.

\begin{theorem}\label{thm:stability HC}
    For every $N\in\N$, there exists  $\delta\coloneqq \delta(N) \in (0,1/2)$ so that the following holds. For every volume vector $\vec v\coloneqq (v_1,...,v_N)$ so that
    \begin{equation}
    \max_{i=1,..,N} |v_i - 1|<\delta,\qquad \sum_{i=1}^Nv_i =N,
    \label{eq:quasi unit volumes}
    \end{equation}
    then 
    \[
        \frac N2 \Per(H) = \inf_G\inf_D \inf \left\{ \sum_i \frac 12 \Per(E_i) \colon \begin{array}{ll}
        E_i\subseteq D, &|E_i|=v_i \\
         |E_i\cap E_j|=0, &\big|D\setminus \cup_i E_i\big|=0, \end{array}\right\}.
    \]
   In particular, the triple $G_{\sf HC}, D(\vec v)$, $ H(v_i)$ is a minimizer.
\end{theorem}

\begin{proof}
    We argue by contradiction. Suppose that there exist $N\in\N$ and sequences $v_i^n \in \big(1/2,3/2\big)$, with $i=1,...,N$, 
    such that $\lim_{n}v_i^n=1$ and
    \begin{equation}
         \inf_G\inf_D \inf \left\{ \sum_i \frac 12 \Per(E_i) \colon \begin{array}{ll}
        E_i\subseteq D, &|E_i|=v^n_i \\
         |E_i\cap E_j|=0, &\big|D\setminus \cup_i E_i\big|=0, \end{array}\right\}< \frac N2 \Per(H).  \label{eq:competitor honeycomb}
\end{equation}
    \noindent{\sc Step 1: Minimizers converge to hexagons in $L^1$}. Let us consider, for each $n\in\N$, the minimizers $G^n,D^n,(E_i)^n$ of \eqref{eq:competitor honeycomb} given by Proposition \ref{prop:planar local finiteness}, satisfying
    \[
     |E_i^n|=v_i^n,\qquad \qquad E_i^n\subseteq D^n \subseteq B_{R_n}(0),   
    \]
    for all $i$, where $R_n = 3 r_{G_n} + C_n$ is as in Proposition \ref{prop:planar local finiteness}. We claim that $R_n$ is uniformly bounded. Indeed, $d(G^n)=N$ for all $n\in\N$, hence the sequence of covering radii $r_{G_n}$ are uniformly bounded as well and, tracing $C_n$ in the proof of Proposition \ref{prop:planar local finiteness}, we see that, having relied on \eqref{eq:C mu N}, it satisfies 
    \[
    C_n \le C(N)\Per( Q ) <\infty,\qquad \forall n \in\N,
    \]
   where $Q\subseteq\R^2$ is the standard cube with volume $N$ (here, the relevant fact is that $\sqrt N\cdot \Z^2$ is an admissible lattice in the minimization).

   Given the condition $d(G^n)=N$ and the uniform bounds, up to a non-relabelled subsequence, we infer the existence of a lattice $G$ with $d(G^n)=N$ and sets $(E_i)\subseteq \R^2$ so that
   \begin{align*}
       &G^n \to G\qquad \text{in Kuratowski sense,}\\
       &E_i^n \to E_i,\quad\text{ in }L^1\text{ for all }i=1,...,N, 
   \end{align*}
   as $n\uparrow\infty$. The second claim is simply given by the BV-precompactness of sets of finite perimeter, as they are uniformly contained in some ball.  Notice then, by $L^1$-convergence and since $v_i^n \to 1$, we also get that 
   \[
   |E_i|=1,\qquad \text{ for all }i=1,...,N, 
   \]
   and thus, since $|E_i^n\cap E^j_n|=0$ for $i\neq j$,  defining $D\coloneqq \cup_i E_i$ we have $|D|=N$. The first claim follows instead by the pre-compactness results for sequences of lattices with fixed volume studied in \cite[Theorem 2.10]{CesaroniFragalaNovaga23}: the case in which $G$ contains a line is ruled out by the fact that $D^n \to D$ and $|D|=N$. In particular, by lower semicontinuity, we see that
   \[
        | (D + g_l) \cap (D+g_m)|\le \liminf_{n \uparrow\infty}|(D^n + g^n_l) \cap (D^n+g^n_m)|=0,\qquad\forall g_l,g_m \in G,
   \]
   where $g^n_l,g_m^n \in G_n$ are converging to $g_l,g_m$ respectively, and similarly that $|\R^2 \setminus \cup_{g\in G} (D+g)| =0$. Thus, we get that $D$ is a fundamental domain for the action of $G$. Combining all, we achieved by lower semicontinuity and by \eqref{eq:competitor honeycomb} that
   \begin{equation}
        \frac 12 \sum_i \Per(E_i) \le  \liminf_{n\uparrow\infty}\frac12 \sum_i\Per(E_i^n) \le \frac N2 \Per(H).
   \label{eq:lsc HC}
   \end{equation}
   However, since the regular Honeycomb is the unique minimizer of the partitioning problem with equal area \cite{Hales01}, we directly deduce that $G=G_{\sf HC}, D=D_{\sf HC}$ and the partition $E_i$ is given by $N$ different  $G$-translation of a regular hexagon $H$ which we denote $H_i$ for each $i$. Notice, for future use that, since equality occurs in \eqref{eq:lsc HC} and each term in the sum is lower semicontinuous, then we also get
   \begin{equation}
   E_i^n \to H_i\quad\text{in }L^1,\qquad \Per(E_i^n)\to\Per(H),\quad\text{for all $i=1,..,N$ as }n\uparrow\infty.
   \label{eq: Per Ei to Per H}
   \end{equation}

   \noindent{\sc Step 2: Modified minimizers converge in Hausdorff distance}. Since in general the sequence 
   $E_i^n$ converge to $H_i$ only in $L^1$, we will modify the sets $E_i^n$, without increasing their perimeter, in order to gain Hausdorff convergence.

   Let us recall a general decomposition result for a set $E\subseteq \R^2$ of finite perimeter. By the structure theory in \cite{AmbrosioCasellesMasnouMorel01}, it is possible to consider indecomposable components $E_j$ of $E$ so that $E=\cup_j E_j$ and $\sum_j\Per(E_j)  = \Per(E)<\infty$. Up to reordering the sets, we can suppose that $|E_1| \ge |E_j|$ for all $j\neq 1$. 
   In particular, we get decompositions $E^n_{i,j}$ of $E_i^n$ so that 
   \begin{equation}
        \sum_j \Per(E^n_{i,j}) = \Per(E^n_i),  \qquad \forall n \in \N,\, i=1,...,N\label{eq:decomposition E i j n}.
   \end{equation}
   Since the sets $E_i^n$ satisfy \eqref{eq: Per Ei to Per H}, it follows that $E^n_{i,1} \to H_i$ and $E^n_{i,j} \to \emptyset$ for $j>1$ in $L^1$, as $n\to+\infty$. As a consequence, letting $F^n_{i,\ell}$ be the bounded indecomposable components of the complementary set  $\R^2\setminus E_i^n$, we can assume that $|E^n_{i,1}|>1/2$, $|E^n_{i,j}|<1/2$ for all $j>1$ and $|F^n_{i,\ell}|<1/2$ for all $\ell$.
   
   Fix now $n\in\N$ and recall that, being $(E^n_i)$ a bounded minimal partition (cf. Proposition \ref{prop:planar local finiteness}), then the regularity theory of Theorem \ref{thm:regularity constrained planar} applies giving in particular that they are regular. Hence, $|E^n_{i,j}| \neq 0$ for at most finitely many indexes $j$. Moreover, $D^n$ are uniformly bounded, and we can thus invoke iteratively Lemma \ref{lem:modification} several but finitely many times until we reach a partition $(\hat E^n_i)$ of simple sets and an associated fundamental domain $\hat D^n$ for the action of $G^n$ satisfying
   \[
    \frac12 \sum_i \Per(\hat E^n_i)\le \frac 12 \sum_i\Per(E_i^n) \le \frac N2\Per(H).
   \]
   By the very construction, (recall that each modification step is made by a piece $E^n_{j,k} \to \emptyset$ in $L^1$ for some $j,k>1$)  we have that $\hat E^n_{i} \to H_i$ as $n\to\infty$ which, combined with the above and by lower semicontinuity, guarantees also that $\Per(\hat E^m_i)\to\Per(H)$ as $n\to\infty$. We are thus in the position to apply Lemma \ref{lem:Hausdorff conv} below to deduce that $\hat E_i^n \to H_i$ in the Hausdorff sense.

\noindent{\sc Step 3: Conclusion}. Under the action of the group $G^n$, the sets $(\hat E_i^n)$ generate a partition of $\R^2$
which is as close as we want in the Hausdorff distance, for $n$ big enough, to the honeycomb partition.
We can thus invoke the local minimality result in \cite{pludaPozzetta23} (see also \cite[Corollary 7.5]{PluTon24}) to deduce that
\[
\frac N2\Per(H) \le \frac 12\sum_i\Per(\hat E_i^n)\qquad\text{for $n$ large enough}.
\]
However, we assumed that strict inequality holds in \eqref{eq:competitor honeycomb}, from which we get the contradiction
\[
\frac 12\sum_i\Per(\hat E_i^n) \le \frac 12\sum_i\Per( E_i^n) < \frac N2\Per(H),
\]
for $n$ sufficiently large. Finally, for the last conclusion, we have that $G_{\sf HC}, D(\vec v)$ and $E_i\coloneqq H(v_i) $ is a minimizer, being a competitor with the lowest possible energy.
\end{proof}
The last assertion states that the configuration in Figure \ref{fig: D with H v} is minimal but it says nothing about uniqueness of minimizers. 
Notice that reordering the positions of the cell in Figure \ref{fig: D with H v} gives trivially other minima. Hence, uniqueness is eventually to be discussed up to permutations of the entries in $(v_1,...,v_N)$.

In the above proof, we used the following two results, namely a modification lemma and an elementary improved convergence lemma. Recall that a set $E$ is  \emph{simple}, if it is indecomposable and saturated, that is, $\R^2\setminus E$ is also indecomposable (see \cite{AmbrosioCasellesMasnouMorel01}). We also denote by $n_E \in \N\cup\{\infty\}$ the number of connected components of $E$.
\begin{lemma}\label{lem:modification}
Let $N\in\N$, let $G$ be a lattice in $\R^2$, let $D\subseteq\R^2$ be a fundamental domain for the action of $G$ and let $(E_i)$ be a partition of $D$ up to negligible sets with $i=1,...,N$. Consider $(E_{i,k})_{ k \in \N}$ indecomposable decomposition of $E_i$ ordered so that $k\mapsto |E_{i,k}|$ is non-increasing and suppose that $(E_i)$ are not all simple sets. Assume further that, if $E_{i,1}$ is not saturated, then $ |  (E_{j,1}+g) \cap  {\rm sat}(E_{i,1})| =0$ for all $g\in G$, $j\ne i$.

Then, there are $i,j \in \{1,...,N\}$ with $i \neq j$ and $k>1,g \in G$ s.t. $\HH^1(\partial^* E_{i,1} \cap\partial^* (E_{j,k}+g))>0$. In particular, setting\[
\hat E_i \coloneqq E_i \cup (E_{j,k}+g),\qquad \hat E_j \coloneqq E_j \setminus E_{j,k}, \qquad \hat E_l \coloneqq E_l,\quad \forall l\neq\{i,j\},
\]
we have that $\hat D \coloneqq \cup_i \hat E_i$ is a fundamental domain for $G$ and $ \sum_i \Per(\hat E_i) <\sum_i \Per(E_i)$. Finally, we have
\[
n_{\hat E_j} \le n_{E_j}-1,\qquad n_{\hat E_l}\le n_{E_l},\quad \forall l\neq j.
\]
\end{lemma}

\begin{proof}
The assumption that $(E_i)$ are not all simple sets guarantees that at least one of the following conditions hold: 
$E_i$ is not saturated for some $i\in \{1,...,N\}$, or $|E_{j,k}|>0$ for some $k>1$ and some $j \in \{1,...,N\}$.

Suppose first that $E_i$ is indecomposable (in particular, $E_{i,1}=E_i$) and it has a hole. Then, since $(E_i)$ is a partition of $D$ and $D$ is a fundamental domain, there is necessarily $E_{j,k}$ for some $j, k\ge 1$ and there is $g\in G$ so that $\HH^1(\partial^* E_i \cap \partial^*( E_{j,k} +g) )>0$. Notice that, by assumption, in this case, we must have $k>1$.

Suppose now there is $E_i$ which is decomposable. If $E_{i,1}$ has a hole, we can repeat the above argument to conclude similarly. 

It remains to consider the case where $E_i$ that is decomposable, in particular $|E_{i,2}|>0$, and $E_{i,1}$ is saturated. 
Define $D_1 \coloneqq \cup_k E_{k,1}$ and notice that $E_{i,2} \subseteq D\setminus D_1$. Since $D$ is a fundamental domain, 
there are $\ell \neq j$ and $g \in G$ so that $ \HH^1(\partial^* (E_{\ell,1} - g)\cap \partial^*E_{j,2})>0$. 

In all cases, we found $i,j \in \{1,...,N\}$, $k>1$ and $g\in G$ so that  $\HH^1(\partial^* E_i \cap \partial^*( E_{j,k} +g) )>0$. Consequently, the sets $\hat E_i$ are well defined for all $i \in \{1,...,N\}$. It is immediate to see that $\hat D$ is again a fundamental domain and that the claimed inequalities for $n_{\hat E_i}$ hold true.
\end{proof}

\begin{lemma}\label{lem:Hausdorff conv}
    Let $B\subseteq \R^2$ be a ball and let $E_n \subseteq B$ be a sequence of simple sets of finite perimeter so that $E_n \to E$ in $L^1$ and $\Per(E_n)\to\Per(H)$ as $n\uparrow\infty$. Then, $E_n^{(1)}$ (the essential interior) converges in the Hausdorff sense to $H$.
\end{lemma}
\begin{proof}
    Since $E_n$ are simple sets of finite perimeter, they are the interior (up to Lebesgue negligible sets) of a Jordan boundary $\Gamma_n$ \cite{AmbrosioCasellesMasnouMorel01}, hence
    \[
    \Per(E_n) = \HH^1(\Gamma_n),\qquad\forall n \in\N.
    \]
    Consider Lipschitz parametrizations $\gamma_n$ of $\Gamma_n$ by arc-length that, by assumptions, are equi-bounded, since $\Gamma_n \subseteq B_R(0)$, and equi-Lipschitz, since ${\rm lenght}(\gamma_n) = \HH^1(\Gamma_n) \to \Per(E)$. In particular, by Ascoli-Arzel\'a it holds up to a subsequence $(n_k)$ that $\gamma_{n_k} \to \gamma$ uniformly as $k\uparrow\infty$ for some rectifiable curve $\gamma$ and, by lower semicontinuity of the length functional we also have
    \begin{equation}
        {\rm length}(\gamma)\le \lim_{n\uparrow\infty}  \HH^1(\Gamma_n) = \Per(H).
    \label{eq:lsc length}
    \end{equation}
    We claim that ${\rm Im}(\gamma)=\partial H$. First, we notice that $\partial H \subseteq {\rm Im}(\gamma)$. Indeed, the assumptions guarantee that $\HH^1\mres{\Gamma_n}\rightharpoonup \HH^1\mres{\partial H}$ weakly in duality with continuous and bounded functions as $n\uparrow\infty$ (here, we are using that $\Gamma_n$ and $\partial^*E_n$ are $\HH^1$-equivalent, the representation formula for the Perimeter measure and that $E_n\to E$ strictly in BV gives the weak convergence of the perimeter measures). Therefore, for every $x \in \partial H$, this convergence guarantees the existence of $x_{n_k} \in \Gamma_{n_k}$  so that $x_{n_k} \to x$ as $k \uparrow\infty$. As $x_{n_k} \in {\rm Im}(\gamma_{n_k})$ and $\gamma_{n_k}$ converges uniformly to $\gamma$, it then holds that $x \in {\rm Im}(\gamma)$ as desired. However, this gives $\Per(H)\le {\rm length}(\gamma)$ that, combined with lower semicontinuity \eqref{eq:lsc length} forces $\Per(H) =  {\rm length}(\gamma)$. This implies the claim.

    We notice that $\gamma$ does not depend on the chosen subsequence hence we get that $\gamma_n \to \gamma$ as $n\uparrow\infty$ along the original sequence. Finally, Hausdorff convergence of the essential interiors $E_n^{(1)}$ to $H$ follows easily from the derived uniform convergence.  
\end{proof}
\begin{remark}
    \rm We observe that Theorem \ref{thm:stability HC} extends to the case of the anisotropic perimeter $\Per_\varphi$ when the associated Wulff shape is a regular Hexagon $H$ (w.l.o.g. we can assume that $H$ has inradius $1$). Indeed, in this case it holds $\Per_\varphi(H) = \Per(H)$ and $\Per_\varphi(E) \ge \Per(E)$ for every $E\subseteq \R^2$ Borel, where $\Per$ is the classical (isotropic) perimeter. Hence, for $\delta(N)$ as in Theorem \ref{thm:stability HC}, and for all $\vec v = (v_1,..,v_N)$ satisfying \eqref{eq:quasi unit volumes}, 
    we deduce that
    \[
    \frac N2\Per_\varphi(H) = \frac N2\Per(H) 
    \le \inf_G\inf_D \inf \left\{ \sum_i \frac 12 \Per_\varphi(E_i) \colon \begin{array}{ll}
        E_i\subseteq D, &|E_i|=v_i \\
         |E_i\cap E_j|=0, &\big|D\setminus \cup_i E_i\big|=0, \end{array}\right\}.
    \]
    Again, we see that \eqref{eq:competitor H} are competitors for the above optimization problem and it is immediate to see that $\Per_\varphi(H(v_i))=\Per(H(v_i))$ for each $i=1,...,N$. Observing that $\sum_i \Per(H(v_i))=N\Per(H)$
    if the hexagons of $H(v_i)$ have the same orientation of $H$,
    we get that the thesis of Theorem \ref{thm:stability HC} also holds in this case.
\end{remark}
\begin{remark}
    \rm 
    Let us comment on the asymptotic behavior of $\delta(N)$ for $N$ large. We notice that $\liminf_{N\to\infty}\delta(N)= 0$. Indeed, 
    for any fixed $\delta\in (0,1/2)$ we consider the lattice $G=N\cdot\mathbb Z^2$
    and the fundamental domain $D:=[0,N]\times[0,N]$. We assume for simplicity that $N$ is even and
    we fix volumes $v_i=1-\delta$ for $i\in \{1,\ldots N^2/2\}$
    and $v_i=1+\delta$ for $i\in \{N^2/2 + 1,\ldots N^2\}$.
    We can write $D=D_1\cup D_2$ with $D_1:=[0,(1-\delta)N/2]\times [0,N]$ and $D_2=[(1-\delta)N/2,N]\times [0,N]$,
    so that $|D_1|=(1-\delta)N^2/2$ and $|D_2|=(1+\delta)N^2/2$.
    We now consider minimal partitions of $D_1$ and $D_2$ into $N^2/2$ sets of volume $1-\delta$ and $1+\delta$, respectively.
    One can show by a direct construction (see for instance \cite{HM05}) that the energy of the resulting partition of $D$ 
    is bounded above by the quantity
    \[
    (\sqrt{1-\delta}+\sqrt{1+\delta}) P(H) N^2/4 + C N\le P(H) N^2/2 - c \delta^2 N^2 + C N, 
    \]
    for suitable constants $c, C>0$, independent of $N$. It follows the partition constructed in Theorem \ref{thm:stability HC}
    can be a minimizer only if $\delta(N^2)\le C/(cN)$. 
\end{remark}
\begin{remark}
    \rm 
    A related problem which we do not consider in this paper is understanding the shape of minimizers, in the planar case and with the classical perimeter, when one volume satisfies $v_1 \approx 1 $ and the remaining ones are very small. At least when $N\in \{2,3\}$, there are natural minimizing candidates which can be obtained as slight modifications of the Honeycomb tiling, see Figure \ref{fig:conjecture}.
\begin{figure}[!ht]
    \centering
    \includegraphics[scale =1.1]{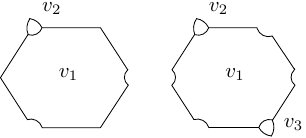}
    \caption{Candidate minimzers for $N=2$ and $N=3$, respectively.}
    \label{fig:conjecture}
\end{figure}
    We mention that the candidate for $N=2$ in Figure \ref{fig:conjecture} has also been proposed
    in the works \cite{FortesTeixeira01,FortesGranerTeixeira02} (see also \cite{Morgan1999}) where its minimality has been investigated by numerical computations. 
    
    We notice that the regularity result obtained in Theorem \ref{thm:regularity constrained planar} completely describes the interfaces between regions of a minimizer, the only possibility being segments or circular arcs, meeting at triple points with equal angles. 
\end{remark}



\end{document}